\documentclass[12pt]{amsart}
\usepackage{amssymb}
\usepackage{hyperref}
\usepackage{graphicx}
\usepackage{subfigure}

\newtheorem{theorem}{Theorem}[section]
\newtheorem{definition}[theorem]{Definition}
\newtheorem{proposition}[theorem]{Proposition}

\newtheorem{lemma}[theorem]{Lemma}

\begin{document}

\title[Block-modified Wishart matrices]
{Block-modified Wishart matrices and free Poisson laws}

\author{Teodor Banica}
\address{T.B.: Department of Mathematics, Cergy-Pontoise University, 95000 Cergy-Pontoise, France. {\tt teodor.banica@u-cergy.fr}}

\author{Ion Nechita}
\address{I.N.: CNRS, Laboratoire de Physique Th\'eorique, IRSAMC, Universit\'e de Toulouse, UPS, 31062 Toulouse, France. {\tt nechita@irsamc.ups-tlse.fr}}

\subjclass[2000]{60B20}
\keywords{Wishart matrix, Poisson law}

\begin{abstract}
We study the random matrices of type $\tilde{W}=(id\otimes\varphi)W$, where $W$ is a complex Wishart matrix of parameters $(dn,dm)$, and $\varphi:M_n(\mathbb C)\to M_n(\mathbb C)$ is a self-adjoint linear map. We prove that, under suitable assumptions, we have the $d\to\infty$ eigenvalue distribution formula $\delta m\tilde{W}\sim\pi_{mn\rho}\boxtimes\nu$, where $\rho$ is the law of $\varphi$, viewed as a square matrix, $\pi$ is the free Poisson law, $\nu$ is the law of $D=\varphi(1)$, and $\delta=tr(D)$.
\end{abstract}

\maketitle

\section*{Introduction}

A complex Wishart matrix of parameters $(dn,dm)$ is a random matrix of type $W=\frac{1}{dm}GG^*$, where $G$ is a $dn\times dm$ matrix with independent complex $\mathcal N(0,1)$ entries. In the limit $d\to\infty$, the eigenvalue distribution of $W$ converges to a certain law $\pi_t$, with $t=m/n$, computed by Marchenko and Pastur in \cite{mpa}. Later on, the free probability theory of Voiculescu \cite{vo1}, \cite{vo2}, \cite{vdn} has led to a new, conceptual point of view on this result: the limiting law $\pi_t$ is in fact the ``free Poisson law'' of parameter $t$.

These fundamental results have been subject to a number of extensions and generalizations. In particular, three types of ``block-modified'' versions of the Marchenko-Pastur theorem appeared in the recent random matrix and free probability literature:
\begin{enumerate}
\item The matrices of type $(1\otimes E)W$, where $E$ is the diagonal matrix formed by the $n$-roots of unity, were investigated some time ago in \cite{bb+}. The limiting laws here are certain compound free Poisson laws, called free Bessel laws.

\item The matrices of type $(id\otimes tr(.)1)W$, where $tr$ is the normalized trace of the $n\times n$ matrices, appeared in connection with the quantum information theory problems investigated in \cite{cn2}. The limiting laws here are the free Poisson laws.

\item The matrices of type $(id\otimes t)W$, where $t$ is the transposition, were investigated by Aubrun in \cite{aub}. His computation, leading to shifted semicircles, was extended in \cite{bne}, where the limiting law was shown to be a difference of free Poisson laws.
\end{enumerate}

The main motivation for the above results comes from quantum information theory. In quantum information theory, the partial transposition map is known to be an ``entanglement witness'': it allows to test if a quantum state (represented by a positive, unit trace matrix) is entangled, in the following sense. If a bipartite quantum state $\rho \in M_d(\mathbb C)\otimes M_n(\mathbb C)$ is separable (i.e. it can be written as a convex combination of product states $\rho^{(1)}_i \otimes \rho^{(2)}_i$), then its partial transposition $\rho^\Gamma=(id\otimes t)\rho$ is also a quantum state. However, if $\rho$ is entangled, then $\rho^\Gamma$ may fail to be positive. In the case where $\rho^\Gamma$ is a positive matrix, the quantum state $\rho$ is said to be PPT (Positive Partial Transpose). Hence, separable states are always PPT and non-PPT states are necessarily entangled. The equivalence of entanglement and non-PPT is known to hold only for total dimension smaller than 6 ($2\times 2$ or $2 \times 3$ product systems) and it fails for larger dimensions, in the sense that there exist PPT entangled states. In the same spirit as in \cite{aub}, the results in \cite{bne} regarding the positivity of the support of the limit measure can be interpreted as results about typicality of PPT states for large quantum systems. The Wishart matrices (normalized to have unit trace) are known to be physically reasonable models for random quantum states on a tensor product $\mathbb C^d \otimes \mathbb C^n$, the parameter $m$ of the Wishart distribution being related to the size of some environment $\mathbb C^{dm}$ needed to define the state. So, as a conclusion, the various technical results in \cite{aub} and \cite{bne}, not to be detailed here, indicate that when $m>2$ and $n<m/4+1/m$, a typical state in $\mathbb C^d \otimes \mathbb C^n$ is PPT.

The starting point of the present work was the following observation: the matrices $W$, $(1\otimes E)W$, $(id\otimes tr(.)1)W$ and $(id\otimes t)W$ appearing in the above considerations are all particular cases of matrices of the form $\tilde{W}=(id\otimes\varphi)W$, where $\varphi:M_n(\mathbb C)\to M_n(\mathbb C)$ is a linear map. So, a natural problem is to try to compute the asymptotic eigenvalue distribution, with $d\to\infty$, of such general ``block-modified'' Wishart matrices $\tilde{W}$.

In this paper we solve this problem, under suitable assumptions on $\varphi$. Let us first recall that the linear maps of type $\varphi:M_n(\mathbb C)\to M_n(\mathbb C)$ are in correspondence with the matrices $\Lambda\in M_n(\mathbb C)\otimes M_n(\mathbb C)$. We will use the following formula for this correspondence:
$$\varphi(A)=(Tr\otimes id)[(t\otimes id)\Lambda\cdot (A\otimes 1)]$$

Here $Tr$ and $t$ denote the usual trace and transposition of the $n\times n$ matrices. This correspondence, when restricted to the subclasses of completely positive maps, is known in quantum information theory as the Choi-Jamiolkowski isomorphism. See \cite{bzy}.

Given a real positive measure $\mu$, not necessarily of mass 1, we denote by $\pi_\mu$ the corresponding compound free Poisson law. Also, we write $\tilde{W}\sim\mu$ if the $d\to\infty$ asymptotic eigenvalue distribution of $\tilde{W}$ is $\mu$.  With these notations, our main result is as follows:

\medskip

\noindent {\bf Theorem.} {\em Let $\tilde{W}=(id\otimes\varphi)W$, where $W$ is a complex Wishart matrix of parameters $(dn,dm)$, and where $\varphi:M_n(\mathbb C)\to M_n(\mathbb C)$ is a self-adjoint linear map, coming from a matrix $\Lambda\in M_n(\mathbb C)\otimes M_n(\mathbb C)$. Then, under suitable ``planar'' assumptions on $\varphi$, we have $\delta m\tilde{W}\sim\pi_{mn\rho}\boxtimes\nu$, with $\rho=law(\Lambda)$, $\nu=law(D)$, $\delta=tr(D)$, where $D=\varphi(1)$.}

\medskip

This result generalizes the above-mentioned computations in \cite{aub}, \cite{bne}, \cite{cn2}, \cite{mpa}. The proof is quite standard, first by applying the Wick formula, then by letting $d\to\infty$ and by using Biane's bijection in \cite{bia}, and finally by using Speicher's free cumulants \cite{sp1}.

We should mention that the combinatorics of some very general ``modified Wishart'' matrices was already investigated some time ago by Graczyk, Letac and Massam \cite{glm}. Our assumption that the modification is performed blockwise is of course a key one: this makes the whole combinatorics controllable, and leads to the above result.

Let us also mention that, as kindly pointed out to us by the referee, there should be as well a direct proof for the above result. Indeed, let $e_{ij}$ be the matrix units of the inclusion $M_n(\mathbb C)\subset M_{nd}(\mathbb C)$. It is known from Voiculescu's work \cite{vo3} that $e_{ij}$ is asymptotically from $W$. Now the matrix $\tilde{W}$ can be viewed as a linear combination of matrices of type $e_{ij}We_{kl}$, so, in principle, combinatorial computations of free probabilistic flavour, in the lines of the book by Nica-Speicher \cite{nsp} for example, should probably do as well the job. Several proofs of this type, for related results, have been worked out, e.g. in \cite{agt}, \cite{bry}, \cite{len}, \cite{oso}. 

We assume in the above statement that $\varphi$ is self-adjoint, in the sense that it maps self-adjoints to self-adjoints. It would be interesting to understand as well the non-self-adjoint case, as to cover the ``free Bessel'' computation in \cite{bb+}, where $\varphi(A)=EA$, with $E$ being the diagonal matrix formed by the $n$-roots of unity. We should mention that our formula $\delta m\tilde{W}\sim\pi_{mn\rho}\boxtimes\nu$ covers as well this situation, but only in the formal sense of \cite{bb+}.

Some other problems concern a possible relation with the complex reflection groups. Indeed, the exact planarity assumptions needed on $\varphi$, not to be detailed here, concern certain sets of partitions $P_{even}(2k,2k)$, for which we refer to Definition 4.4 below, and these sets are known to span the centralizer spaces for the hyperoctahedral group \cite{bbc}.

The above questions are probably all related. A common answer to them should probably come from a very general random matrix/free probability formula, extending our present formula $\delta m\tilde{W}\sim\pi_{mn\rho}\boxtimes\nu$. But we have no further results here.

The paper is organized as follows: in 1-2 we perform a joint study of the compound free Poisson laws and of the block-modified Wishart matrices, and in 3-4 we develop a number of supplementary ingredients, and we state and prove our main result.

\subsection*{Acknowledgements}

We would like to thank Mireille Capitaine, Beno\^it Collins, Maxime F\'evrier and Camille Male for several useful discussions, and the anonymous referee of this paper for a number of useful comments and suggestions. The work of T.B. was supported by the ANR grant ``Granma''. I.N. acknowledges financial support from the ANR project OSvsQPI 2011 BS01 008 01 and from a CNRS PEPS grant.
\section{Poisson laws}

Our starting point is the Marchenko-Pastur theorem \cite{mpa}. Let us recall that a complex Wishart matrix of parameters $(N,M)$ is a random $N\times N$ matrix of type $W=\frac{1}{M}GG^*$, where $G$ is a $N\times M$ matrix with independent complex Gaussian $\mathcal N(0,1)$ entries. 

\begin{theorem}
In the limit $N,M\to\infty$, $M/N\to t\geq 0$, the law of $tW$ converges to
$$\pi_t=\max (1-t,0)\delta_0+\frac{\sqrt{4t-(x-1-t)^2}}{2\pi x} 1_{[(\sqrt{t}-1)^2, (\sqrt{t}+1)^2]}(x)\,dx$$
which is called Marchenko-Pastur law of parameter $t\geq 0$.
\end{theorem}

\begin{proof}
This follows for instance by checking, by using the Wick formula, that the asymptotic moments of $W$ coincide with those of $\pi_t$. We refer here the reader to section 2 below, where we present a generalization of this theorem, along with a complete proof.
\end{proof}

In order to deal with the block-modified case, we will need a free probability point of view on $\pi_t$. The idea is that $\pi_t$ appears naturally via a  ``free Poisson limit'' procedure.

We recall from \cite{vdn} that a noncommutative probability space is a pair $(A,\varphi)$, where $A$ is a unital $C^*$-algebra, and $\varphi:A\to\mathbb C$ is a positive unital trace. The law of a self-adjoint element $a\in A$ is the probability measure on the spectrum of $a$ (which is a compact subset of $\mathbb R$) given by $\int f(x)d\mu(x)=\varphi(f(a))$, for any continuous function $f:\mathbb R\to\mathbb C$.

Two subalgebras $B,C\subset A$ are called free if $\varphi(\ldots b_ic_ib_{i+1}c_{i+1}\ldots)=0$ whenever $b_i\in B$ and $c_i\in C$ satisfy $\varphi(b_i)=\varphi(c_i)=0$. Two elements $b,c\in A$  are called free whenever the algebras $B = \langle b \rangle $ and $C = \langle c \rangle$ that they generate are free. Finally, the free convolution operation $\boxplus$ is defined as follows: if $\mu,\nu$ are compactly supported probability measures on $\mathbb R$, then $\mu\boxplus\nu$ is the law of $b+c$, where $b,c$ are free, having laws $\mu,\nu$. See \cite{vdn}.

With these definitions, we have two conceptual results about $\pi_t$, as follows.

\begin{proposition}
We have $\pi_{s+t}=\pi_s\boxplus\pi_t$, for any $s,t\geq 0$, so that the Marchenko-Pastur laws form a semigroup with respect to Voiculescu's free convolution operation.
\end{proposition}

\begin{proof}
We recall from \cite{vo1} that the operation $\boxplus$ for real probability measures is linearized by the $R$-transform, constructed as follows: first, we let $f(y)=1+m_1y+m_2y^2+\ldots$ be the moment generating function of our measure, so that $G(\xi)=\xi^{-1}f(\xi^{-1})$ is the Cauchy transform; then we set $R(y)=K(y)-y^{-1}$, where $K(y)$ is such that $G(K(y))=y$.

By Stieltjes inversion, the Cauchy transform of $\pi_t$ is given by:
$$G(\xi)=\frac{(\xi+1-t)+\sqrt{(\xi+1-t)^2-4\xi}}{2\xi}$$

Thus we can compute the $R$-transform, by proceeding as follows:
\begin{eqnarray*}
\xi G^2+1=(\xi+1-t)G
&\implies&Ky^2+1=(K+1-t)y\\
&\implies&Ry^2+y+1=(R+1-t)y+1\\
&\implies&R=t/(1-y)
\end{eqnarray*}

Now since the expression $t/(1-y)$ is linear in $t$, this gives the result.
\end{proof}

\begin{theorem}
We have the free Poisson limit formula
$$\pi_t=\lim_{n\to\infty}\left(\left(1-\frac{t}{n}\right)\delta_0+\frac{t}{n}\delta_1\right)^{\boxplus n}$$
so that $\pi_t$ is called as well ``free Poisson law'' of parameter $t\geq 0$.
\end{theorem}

\begin{proof}
This result is once again well-known, see e.g. Speicher \cite{sp2} or Hiai and Petz \cite{hpe}, and is for instance a particular case of Theorem 1.5 below.
\end{proof}

In what follows we will need some generalizations of the above results. The class of ``compound free Poisson laws'' was introduced by Speicher in \cite{sp2}, studied by Hiai and Petz in \cite{hpe}, then further studied in \cite{bb+}, \cite{bsk}. Since we will be only interested in the ``discrete case'', it is technically convenient to introduce these distributions as follows.

\begin{definition}
For $\mu=\sum_{i=1}^sc_i\delta_{z_i}$ with $c_i>0$ and $z_i\in\mathbb C$, we let
$$\pi_\mu={\rm law}\left(\sum_{i=1}^sz_i\alpha_i\right)$$
where the variables $\alpha_i$ are free Poisson of parameter $c_i$, free.
\end{definition}

Observe that we don't necessarily assume $\mu$ to be of mass $1$. Observe also that we don't assume the numbers $z_i$ to be distinct: the fact that we can indeed do so comes from Proposition 1.2, which shows that the distribution $\pi_\mu$ is indeed well-defined.

\begin{theorem}
If $\mu$ is real we have the Poisson limit formula
$$\pi_\mu=\lim_{n\to\infty}\left(\left(1-\frac{c}{n}\right)\delta_0+\frac{1}{n}\mu\right)^{\boxplus n}$$
where $c=mass(\mu)$, so that $\pi_\mu$ is called ``compound free Poisson law'' associated to $\mu$.
\end{theorem}

\begin{proof}
This result is from \cite{bsk}, we present below the idea of the proof. Let $\rho_n$ be the measure appearing in the statement, under the convolution sign. We have:
$$G_{\rho_n}(\xi)=\left(1-\frac{c}{n}\right)\frac{1}{\xi}+\frac{1}{n}\sum_{i=1}^s\frac{c_i}{\xi-z_i}$$

Now since $K_{\rho_n}(y)=y^{-1}+R_{\rho_n}(y)=y^{-1}+R/n$, where $R=R_{\rho_n^{\boxplus n}}(y)$, we get:
$$y=\left(1-\frac{c}{n}\right)\frac{1}{y^{-1}+R/n}+\frac{1}{n}\sum_{i=1}^s\frac{c_i}{y^{-1}+R/n-z_i}$$

Now multiplying by $n/y$, rearranging the terms, and letting $n\to\infty$, we get:
$$\frac{c+yR}{1+yR/n}=\sum_{i=1}^s\frac{c_i}{1+yR/n-yz_i}
\implies R_{\pi_\mu}(y)=\sum_{i=1}^s\frac{c_iz_i}{1-yz_i}$$

On the other hand, let $\alpha$ be the sum of free Poisson variables in the statement. By using the $R$-transform formula in the proof of Proposition 1.2, we have:
$$R_{\alpha_i}(y)=\frac{c_i}{1-y}
\implies R_{z_i\alpha_i}(y)=\frac{c_iz_i}{1-yz_i}
\implies R_\alpha(y)=\sum_{i=1}^s\frac{c_iz_i}{1-yz_i}$$

Thus we have indeed the same formula as above, and we are done.
\end{proof}

Finally, we will need the notion of free cumulant, introduced by Speicher in \cite{sp1}. The free cumulants $\kappa_p(a)$ of a self-adjoint variable $a$ are the coefficients of its $R$-transform:
$$R_a(y)=\sum_{p=0}^\infty\kappa_{p+1}(a)y^p$$

With $\kappa_\pi(a)=\prod_{b\in\pi}\kappa_{\# b}(a)$, where the product is over all blocks of $\pi$, and $\#$ is the size of blocks, we have then the following moment-cumulant formula, due to Speicher \cite{sp1}:
$$\varphi(a^p)=\sum_{\pi\in NC(p)}\kappa_\pi(a)$$

For concrete applications, it is rather the converse statement that we will use:  whenever we have a sequence of numbers $\kappa_p(a)$ such the above moment formula holds for any $k\in\mathbb N$, it follows that these numbers $\kappa_p(a)$ are the free cumulants of $a$, see \cite{sp1}.

The following will be our main tool for detecting compound free Poisson laws.

\begin{theorem}
If $\mu$ is real then the free cumulants of $\pi_\mu$ are the moments of $\mu$.
\end{theorem}

\begin{proof}
We write $\mu=\sum_{i=1}^sc_i\delta_{z_i}$ with $c_i>0$ and $z_i\in\mathbb R$. We know from the proof of Theorem 1.5 that the $R$-transform of $\pi_\mu$ is:
$$R(y)=\sum_{i=1}^s\frac{c_iz_i}{1-yz_i}=\sum_{i=1}^s\sum_{p=0}^\infty c_iz_i^{p+1}y^p=\sum_{p=0}^\infty\left(\sum_{i=1}^sc_iz_i^{p+1}\right)y^p$$

Now since the free cumulants are the coefficients of $R$, this gives the result.
\end{proof}

\section{Wishart matrices}

Consider the embedding $NC(p)\subset S_p$ obtained by ``cycling inside each block''. That is, each block $b=\{b_1,\ldots,b_k\}$ with $b_1<\ldots<b_k$ of a given noncrossing partition $\pi\in NC(p)$ produces by definition the cycle $(b_1\ldots b_k)$ of the corresponding permutation $\pi\in S_p$. 

Observe that the number of blocks of $\pi\in NC(p)$ corresponds in this way to the number of cycles of the corresponding permutation $\pi\in S_p$. This number will be denoted $|\pi|$.

For $\pi\in NC(p)$ we denote by $||\pi||$ the number of blocks of $\pi$ having even size. We will need a number of facts on partitions, summarized in the following statement:

\begin{lemma}
If $\gamma\in NC(p)$ is the one-block partition, then $|\pi|+|\pi\gamma^{-1}|\leq p+1$, with equality iff $\pi\in NC(p)$. In addition, for $\pi\in NC(p)$ we have:
\begin{enumerate}
\item $|\gamma^{-1}\pi|=|\pi\gamma^{-1}|=p+1-|\pi|$.

\item $|1|+|\pi^l\gamma|=|\pi\gamma^{-1}|+|\pi^{l+1}|$, for any $l\in\mathbb Z$.

\item $|\gamma\pi|=|\pi\gamma|=||\pi||+1$.
\end{enumerate}
\end{lemma}

\begin{proof}
Observe first that $\gamma\in S_p$ is the full cycle, $\gamma=(1\ldots p)$. Let also $1\in NC(p)$ be the $p$-block partition, so that the associated permutation $1\in S_p$ is the identity.

It is known that $l(\pi)=p-|\pi|$ is the length of permutations, so that $d(\pi,\sigma)=l(\pi\sigma^{-1})$ is the usual distance on $S_p$. Now the triangle inequality $d(1,\gamma)\leq d(1,\pi)+d(\pi,\gamma)$ reads $p-1\leq (p-|\pi|)+(p-|\pi\gamma^{-1}|)$, so we have $|\pi|+|\pi\gamma^{-1}|\leq p+1$, as claimed. For the assertion regarding the case where we have equality, see Biane \cite{bia}.

(1) This is clear from the first assertion, or from the well-known fact that $\gamma^{-1}\pi,\pi\gamma^{-1}$ have the same cycle structure as the right and left Kreweras complements of $\pi$.

(2) By using $|1|=p$ and $|\pi\gamma^{-1}|=p+1-|\pi|$, we must prove the following formula:
$$|\pi^l\gamma|-1=|\pi^{l+1}|-|\pi|$$

Observe first that this formula holds for $\pi=\gamma$. We will prove this formula by recurrence on the number of blocks of $\pi$. Since $\pi$ is noncrossing, it is enough to check the stability of the above formula by ``concatenation'', $\pi=(\pi_1,\pi_2)$. Since the right term of the above formula is additive with respect to concatenation, we have to prove that the left term is additive as well with respect to concatenation, i.e. we have to prove that:
$$|(\pi_1,\pi_2)^l\gamma|-1=(|\pi_1^l\gamma_1|-1)+(|\pi_2^l\gamma_2|-1)$$

With $\sigma_1=\pi_1^l$ and $\sigma_2=\pi_2^l$, the formula to be proved becomes:
$$|(\sigma_1,\sigma_2)\gamma|-1=(|\sigma_1\gamma_1|-1)+(|\sigma_2\gamma_2|-1)$$

In order to prove this latter formula, we use the following key identity:
$$(\gamma_1,\gamma_2)=\gamma(k,p)$$

More precisely, by using this identity, we have as claimed:
$$|\sigma_1\gamma_1|+|\sigma_2\gamma_2|
=|(\sigma_1,\sigma_2)(\gamma_1,\gamma_2)| 
=|(\sigma_1,\sigma_2)\gamma(k,p)|
=|(\sigma_1,\sigma_2)\gamma|+1$$

Here at right we have used the general fact that $|\pi(a,b)|=|\pi|+1$ when $a,b$ are in the same cycle of $\pi$, and $|\pi(a,b)|=|\pi|-1$ if $a,b$ are in different cycles of $\pi$, and the fact that, in the above situation, $k,p$ must be in the same cycle of $(\sigma_1,\sigma_2)\gamma$.

(3) This is the combinatorial lemma from our previous paper \cite{bne}, which follows as well from (2). Indeed, at $l=1$ the formula in (2) becomes $|1|+|\pi\gamma|=|\pi\gamma^{-1}|+|\pi^2|$, so:
\begin{eqnarray*}
|\pi\gamma|
&=&|\pi\gamma^{-1}|+|\pi^2|-|1|\\
&=&(p+1-|\pi|)+(|\pi|+||\pi||)-p\\
&=&||\pi||+1
\end{eqnarray*}

Together with the fact that $|\pi\gamma|=|\gamma\pi|$, this gives the result.
\end{proof}

Consider now a complex Wishart matrix of parameters $(dn,dm)$, $W=\frac{1}{dm}GG^*$, where $G$ is a $dn\times dm$ matrix with independent complex $\mathcal N(0,1)$ entries. As explained in the introduction, we are interested in the study of the ``block-modified'' versions of $W$, which are defined as follows:

\begin{definition}
Associated to any Wishart matrix $W$ of parameters $(dn,dm)$ and any linear map $\varphi:M_n(\mathbb C)\to M_n(\mathbb C)$ is the ``block-modified'' matrix $\tilde{W}=(id\otimes\varphi)W$.
\end{definition}

In what follows, the idea will be to relate the limiting $d\to\infty$ eigenvalue distribution of $\tilde{W}$ to the usual eigenvalue distribution of $\varphi$, viewed as a square matrix.

In order to view $\varphi$ as a square matrix, we use the following correspondence:

\begin{definition}
Associated to any linear map $\varphi:M_n(\mathbb C)\to M_n(\mathbb C)$ is the square matrix $\Lambda\in M_n(\mathbb C)\otimes M_n(\mathbb C)$ given by
$$\varphi(A)=(Tr\otimes id)[(t\otimes id)\Lambda\cdot (A\otimes 1)]$$
where $Tr$ and $t$ are the usual trace and transposition of the $n\times n$ matrices.
\end{definition}

Observe that the correspondence $\varphi\to\Lambda$ is bijective. We will use as well the correspondence in the other sense: every time we will have a matrix $\Lambda\in M_n(\mathbb C)\otimes M_n(\mathbb C)$, we could speak about the associated linear map $\varphi:M_n(\mathbb C)\to M_n(\mathbb C)$, via the above formula. This correspondence, when restricted to subclasses of (completely) positive maps, is known in quantum information theory as the Choi-Jamiolkowski isomorphism. See \cite{bzy}.

In order to get now more insight into the multiplicativity condition, we use the tensor planar algebra of Jones \cite{jon}, or, equivalently, the diagrammatic formalism in \cite{cn1}. We represent the matrices $\Lambda\in M_n(\mathbb C)\otimes M_n(\mathbb C)$ as vertical boxes with $2$ left legs and $2$ right legs, so that the correspondence $\Lambda\to\varphi$ is the one described by Figure \ref{fig:phi-Lambda}, where we write $\Lambda^\Gamma = (t\otimes id)\Lambda$ for the partially transposed map. 

It will be often convenient to write $\Lambda$ in usual matrix notation:
$$\Lambda=\sum_{abcd}\Lambda_{ab,cd}e_{ac}\otimes e_{bd}$$

With this convention, the associated linear map $\varphi$ is given by:
$$\varphi(A)_{bd}=\sum_{ac}\Lambda_{ab,cd}A_{ac}$$

Observe also that, in terms of the above square matrix $\Lambda$, we have (see also Figure \ref{fig:tildeW-Lambda} for the diagram):
$$\tilde{W}=(id\otimes Tr\otimes id)[(1\otimes (t\otimes id)\Lambda)(W\otimes 1)]$$

\begin{figure}
\centering
\subfigure[]{\label{fig:phi-Lambda}\includegraphics{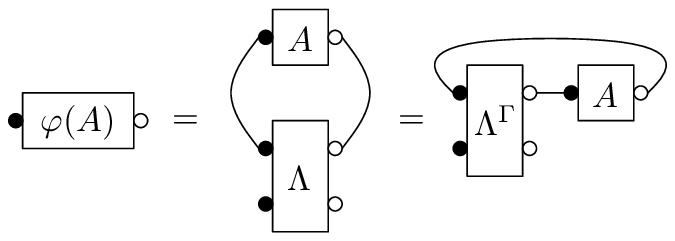}}\qquad
\subfigure[]{\label{fig:tildeW-Lambda}\includegraphics{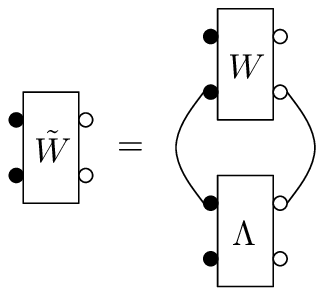}}
\caption{(a) The action of the map $\varphi$ through $\Lambda$, and (b) the block-modified Wishart matrix $\tilde W$ as the action of $\Lambda$ on $W$.}
\end{figure}

Finally, observe that, in matrix notation, the entries of $\tilde{W}$ are given by:
$$\tilde{W}_{ia,jb}=\sum_{ef}\Lambda_{ea,fb}W_{ie,jf}$$

Given a complex linear functional $\mathbb E:A\to\mathbb C$ and a permutation $\pi\in S_p$ we set $\mathbb E_\pi(a)=\sum_{b\in\pi}\mathbb E(a^{\# b})$ for any $a\in A$, where the sum is over all the cycles of $\pi$. We will use this notation throughout the reminder of the paper, often with $\mathbb E$ being the usual trace of matrices, $Tr:M_N(\mathbb C)\to\mathbb C$, or the normalized trace of matrices, $tr:M_N(\mathbb C)\to\mathbb C$.

\begin{theorem}
We have the asymptotic moment formula:
$$\lim_{d\to\infty}(\mathbb E\circ tr)((m\tilde{W})^p)=\sum_{\pi\in NC(p)}(mn)^{|\pi|}tr_{(\pi,\gamma)}(\Lambda)$$
\end{theorem}

\begin{proof}
According to the above formula for the entries of $\tilde{W}$, we have:
\begin{align*}
tr(\tilde{W}^p)
&=(dn)^{-1}\sum_{i_ra_r}\prod_s\tilde{W}_{i_sa_s,i_{s+1}a_{s+1}}\\
&=(dn)^{-1}\sum_{i_ra_re_rf_r}\prod_s\Lambda_{e_sa_s,f_sa_{s+1}}W_{i_se_s,i_{s+1}f_s}\\
&=(dn)^{-1}(dm)^{-p}\sum_{i_ra_re_rf_rj_rb_r}\prod_s\Lambda_{e_sa_s,f_sa_{s+1}}G_{i_se_s,j_sb_s}\bar{G}_{i_{s+1}f_s,j_sb_s}
\end{align*}

The average of the general term can be computed by the Wick rule:
\begin{align*}
\mathbb E\left(\prod_sG_{i_se_s,j_sb_s}\bar{G}_{i_{s+1}f_s,j_sb_s}\right)
&=\#\{\pi\in S_p|i_{\pi(s)}=i_{s+1},\,e_{\pi(s)}=f_s,\,j_{\pi(s)}=j_s,\,b_{\pi(s)}=b_s\}
\end{align*}

Let us look now at the above sum. The $i,j,b$ indices range over sets having respectively $d,d,m$ elements, and they have to be constant under the action of $\pi\gamma^{-1},\pi,\pi$. Thus when summing over these $i,j,b$ indices we simply obtain a $d^{|\pi\gamma^{-1}|}d^{|\pi|}m^{|\pi|}$ factor, so we get:
\begin{align*}
(\mathbb E\circ tr)(\tilde{W}^p)
&=(dn)^{-1}(dm)^{-p}\sum_{\pi\in S_p}d^{|\pi\gamma^{-1}|}(dm)^{|\pi|}\sum_{a_re_r}\prod_s\Lambda_{e_sa_s,e_{\pi(s)}a_{s+1}}\\
&=n^{-1}m^{-p}\sum_{\pi\in S_p}d^{|\pi|+|\pi\gamma^{-1}|-p-1}m^{|\pi|}Tr_{(\pi,\gamma)}(\Lambda)\\
&=m^{-p}\sum_{\pi\in S_p}d^{|\pi|+|\pi\gamma^{-1}|-p-1}(mn)^{|\pi|}tr_{(\pi,\gamma)}(\Lambda)
\end{align*}

By Lemma 2.1, with $d\to\infty$ the sum restricts over $\pi\in NC(p)$, and we are done.
\end{proof}

One can also prove the result using the graphical Wick technique developed in \cite{cn2}. There is an obvious relation between the above result and the combinatorics of the compound free Poisson laws. More precisely, we have the following similar result:

\begin{theorem}
If $\Lambda=\Lambda^*$ then the moments of $\pi_{mn\rho}$, with $\rho=law(\Lambda)$, are:
$$M_p=\sum_{\pi\in NC(p)}(mn)^{|\pi|}tr_{(\pi,\pi)}(\Lambda)$$
\end{theorem}

\begin{proof}
We know from Theorem 1.6 that the free cumulants of $\pi_{mn\rho}$ are the moments of $mn\rho$, given by $\kappa_p=mn\cdot tr(\Lambda^p)=mn\cdot tr_{(\gamma,\gamma)}(\Lambda)$. Together with Speicher's moment-cumulant formula, explained in section 1 above, this gives the result.
\end{proof}

In what follows we will try to exploit the obvious similarity between Theorem 2.4 and Theorem 2.5. We will split our study into 2 parts: the unital case will be investigated in section 3, and the general, non-unital case will be investigated in section 4.

\section{The unital case}

A linear map $\varphi:M_n(\mathbb C)\to M_n(\mathbb C)$ is called self-adjoint if $A=A^*$ implies $\varphi(A)=\varphi(A)^*$. This is the same as asking for the corresponding matrix $\Lambda$ to be self-adjoint.

We use the notation $\mathbb E_\rho$ from the previous section, in the case where $\mathbb E=tr$ is the trace, and $\rho=(\pi,\sigma)$ is a product of permutations, via the embedding $S_p\times S_q\subset S_{pq}$.

\begin{theorem}
Assume that $\varphi$ is self-adjoint. If $tr_{(\pi,\gamma)}(\Lambda)=tr_{(\pi,\pi)}(\Lambda)$ for any $p\in\mathbb N$ and any $\pi\in NC(p)$, then $m\tilde{W}\sim\pi_{mn\rho}$, with $\rho=law(\Lambda)$.
\end{theorem}

\begin{proof}
We know from Theorem 2.4 and Theorem 2.5 that the formula $m\tilde{W}\sim\pi_{mn\rho}$ is equivalent to the following formula, which should be valid for any $p\in\mathbb N$:
$$\sum_{\pi\in NC(p)}(mn)^{|\pi|}tr_{(\pi,\gamma)}(\Lambda)=\sum_{\pi\in NC(p)}(mn)^{|\pi|}tr_{(\pi,\pi)}(\Lambda)$$

Now since in the case $tr_{(\pi,\gamma)}(\Lambda)=tr_{(\pi,\pi)}(\Lambda)$ this formula holds, we are done.
\end{proof}

The point now is that, in all the examples that we have, the above result applies only when $\varphi$ is unital modulo scalars. The non-unital case, which requires a more subtle combination of Theorem 2.4 and Theorem 2.5, will be discussed in the next section.

For a matrix $A\in M_n(\mathbb C)$ we denote by $A^\delta\in M_n(\mathbb C)$ its diagonal.

\begin{theorem}
The formula $tr_{(\pi,\gamma)}(\Lambda)=tr_{(\pi,\pi)}(\Lambda)$ holds for any map which is unital modulo scalars, in the sense that $\varphi(1)=c1$ with $c\in\mathbb C-\{0\}$, of the following form:
\begin{enumerate}
\item $\varphi(A)=Tr(BA)C$.

\item $\varphi(A)=BAC$.

\item $\varphi(A)=BA^tC$.

\item $\varphi(A)=B(EAF)^\delta C$.
\end{enumerate}

\noindent In addition, the set of above maps is stable by composition and inversion, and the formula $tr_{(\pi,\gamma)}(\Lambda)=tr_{(\pi,\pi)}(\Lambda)$ is stable by taking tensor products, and multiplying by scalars.
\end{theorem}

\begin{proof}
If we write $\Lambda=\sum_{abcd}\Lambda_{ab,cd}e_{ac}\otimes e_{bd}$, then $\varphi(A)_{bd}=\sum_{ac}\Lambda_{ab,cd}A_{ac}$. This shows that in the cases (1-3) we have $\Lambda_{ab,cd}=B_{ca}C_{bd},B_{ba}C_{cd},B_{bc}C_{ad}$ respectively.

Let us compute the numbers $Tr_{(\pi,\sigma)}(\Lambda)=\sum_{a_re_r}\prod_s\Lambda_{e_sa_s,e_{\pi(s)}a_{\sigma(s)}}$:
\begin{eqnarray*}
Tr_{(\pi,\sigma)}(\Lambda^1)
&=&\sum_{a_re_r}\prod_sB_{e_{\pi(s)}e_s}C_{a_sa_{\sigma(s)}}=Tr_\pi(B)Tr_\sigma(C)\\
Tr_{(\pi,\sigma)}(\Lambda^2)
&=&\sum_{a_re_r}\prod_sB_{a_se_s}C_{e_{\pi(s)}a_{\sigma(s)}}=Tr_{\pi\sigma^{-1}}(BC)\\
Tr_{(\pi,\sigma)}(\Lambda^3)
&=&\sum_{a_re_r}\prod_sB_{a_se_{\pi(s)}}C_{e_sa_{\sigma(s)}}=Tr_{\sigma\pi}(BC)
\end{eqnarray*}

In the case (4) now, we have $\Lambda_{ab,cd}=\sum_xB_{bx}E_{xa}F_{cx}C_{xd}$, and we get:
$$Tr_{(\pi,\sigma)}(\Lambda^4)
=\sum_{a_re_rx_r}\prod_sB_{a_sx_s}E_{x_se_s}F_{e_{\pi(s)}x_s}C_{x_sa_{\sigma(s)}}
=\sum_{x_r}\prod_s(CB)_{x_sx_{\sigma(s)}}(EF)_{x_{\pi(s)}x_s}$$

Assume now that the maps $\varphi$ are unital modulo scalars, with $C=c1$ in (1), with $BC=c1$ in (2,3), and with $(CB)(EF)^\delta=c1$ in (4). We get that:
\begin{eqnarray*}
tr_{(\pi,\sigma)}(\Lambda^1)
&=&tr_\pi(B)tr_\sigma(C)=c^ptr_\pi(B)\\
tr_{(\pi,\sigma)}(\Lambda^2)
&=&n^{|\pi\sigma^{-1}|-|\pi|-|\sigma|}tr_{\pi\sigma^{-1}}(BC)=n^{|\pi\sigma^{-1}|-|\pi|-|\sigma|}c^p\\
tr_{(\pi,\sigma)}(\Lambda^3)
&=&n^{|\sigma\pi|-|\pi|-|\sigma|}tr_{\sigma\pi}(BC)=n^{|\sigma\pi|-|\pi|-|\sigma|}c^p\\
tr_{(\pi,\sigma)}(\Lambda^4)
&=&n^{-|\pi|-|\sigma|}\sum_{x\leq\ker(\pi\wedge\sigma)}(CB)_{xx}(EF)_{xx}=n^{|\pi\wedge\sigma|-|\pi|-|\sigma|}c^p
\end{eqnarray*}

We claim that we have $tr_{(\pi,\gamma)}(\Lambda)=tr_{(\pi,\pi)}(\Lambda)$ in all cases. Indeed, in case (1) this is clear, and in (2-4) this follows from the following equalities for the exponents of $n$:
\begin{eqnarray*}
(p+1-|\pi|)-|\pi|-1&=&p-|\pi|-|\pi|\\
(||\pi||+1)-|\pi|-1&=&(|\pi|+||\pi||)-|\pi|-|\pi|\\
1-|\pi|-1&=&|\pi|-|\pi|-|\pi|
\end{eqnarray*}

Regarding now the stability properties, the inverses of the maps (1-4) are also of type (1-4). Also, the maps (1-4) are stable by composition, the formula for $i\circ j$ being:
$$\begin{matrix}
i\backslash j&1&2&3&4\\
1&1&1&1&1\\
2&1&2&3&4\\
3&1&3&2&4\\
4&1&4&4&4
\end{matrix}$$

For the tensor product assertion now, assume that $\varphi=\varphi^1\otimes\varphi^2$. We have then $\Lambda=\Lambda^1\otimes\Lambda^2$, so in matrix notation $\Lambda_{a\alpha b\beta,c\gamma d\delta}=\Lambda^1_{ab,cd}\Lambda^2_{\alpha\beta,\gamma\delta}$, and we get:
\begin{eqnarray*}
Tr_{(\pi,\sigma)}(\Lambda)
&=&\sum_{a_r\alpha_re_r\varepsilon_r}\prod_s\Lambda_{e_s\varepsilon_sa_s\alpha_s,e_{\pi(s)}\varepsilon_{\pi(s)}a_{\sigma(s)}\alpha_{\sigma(s)}}\\
&=&\sum_{a_r\alpha_re_r\varepsilon_r}\prod_s\Lambda^1_{e_sa_s,e_{\pi(s)}a_{\sigma(s)}}\Lambda^2_{\varepsilon_s\alpha_s,\varepsilon_{\pi(s)}\alpha_{\sigma(s)}}\\
&=&Tr_{(\pi,\sigma)}(\Lambda^1)Tr_{(\pi,\sigma)}(\Lambda^2)
\end{eqnarray*}

It follows that $tr_{(\pi,\sigma)}(\Lambda)=tr_{(\pi,\sigma)}(\Lambda^1)tr_{(\pi,\sigma)}(\Lambda^2)$, and we are done. Finally, the assertion regarding the multiplication by scalars is clear.
\end{proof}

As a first consequence, we obtain some previously known results, from \cite{cn2}, \cite{mpa}, \cite{bne}:

\begin{proposition}
We have the following results:
\begin{enumerate}
\item $t(id\otimes tr(.)1)W\sim\pi_t$, where $t=mn$.

\item $tW\sim\pi_t$, where $t=m/n$.

\item $m(id\otimes t)W\sim law(\alpha-\beta)$, where $\alpha,\beta$ are free Poisson $(m(n\pm 1)/2)$, free. 

\item $m(id\otimes(.)^\delta)W\sim\pi_m$.
\end{enumerate}
\end{proposition}

\begin{proof}
Our claim is that (1-4) above correspond via Theorem 3.1 to the assertions (1-4) in Theorem 3.2, at $B=C=1$. Indeed, let us see what happens in this case:

(1) Here we have $\varphi(A)=Tr(A)1$, hence $\tilde{W}=(id\otimes Tr(.)1)W$. Also, we have $\Lambda_{ab,cd}=\delta_{ca}\delta_{bd}$, hence $\Lambda=\sum_{ab}e_{aa}\otimes e_{bb}$ is the identity matrix: $\Lambda=1$. Thus $\rho=\delta_1$, so $\pi_{mn\rho}=\pi_{mn}$, so Theorem 3.2 says at $B=C=1$ that we have $m\tilde{W}\sim\pi_{mn}$, as claimed.

(2) Here we have $\varphi(A)=A$, so $\tilde{W}=W$. Also, we have $\Lambda_{ab,cd}=\delta_{ab}\delta_{cd}$, hence $\Lambda=\sum_{ac}e_{ac}\otimes e_{ac}$, so we have $\Lambda=nP$, where $P$ is the rank one projection on the vector $\sum_ae_a\otimes e_a\in\mathbb C^n\otimes\mathbb C^n$. Thus $\rho=\frac{n^2-1}{n^2}\delta_0+\frac{1}{n^2}\delta_n$, so $mn\rho=\frac{m(n^2-1)}{n}\delta_0+\frac{m}{n}\delta_n$, so $\pi_{mn\rho}=law(n\alpha)$, where $\alpha$ is free Poisson $(m/n)$. Thus Theorem 3.2 says at $B=C=1$ that we have $mW\sim law(n\alpha)$, and by dividing by $n$ we get $(m/n)W\sim law(\alpha)$, as claimed.

(3) Here we have $\varphi(A)=A^t$, so $\tilde{W}=(id\otimes t)W$. Also, we have $\Lambda_{ab,cd}=\delta_{bc}\delta_{ad}$, so $\Lambda=\sum_{ac}e_{ac}\otimes e_{ca}$ is the flip: $\Lambda(e_c\otimes e_a)=e_a\otimes e_c$. Thus $\rho=\frac{n-1}{2n}\delta_{-1}+\frac{n+1}{2n}\delta_1$, so $mn\rho=\frac{m(n-1)}{2}\delta_{-1}+\frac{m(n+1)}{2}\delta_1$, so $\pi_{mn\rho}=law(\alpha-\beta)$, where $\alpha,\beta$ are as in the statement. Thus Theorem 3.2 says at $B=C=1$ that we have $m\tilde{W}\sim law(\alpha-\beta)$, as claimed.

(4) Here we have $\varphi(A)=A^\delta$, so $\tilde{W}=(id\otimes(.)^\delta)W$. Also, we have $\Lambda_{ab,cd}=\delta_{a,b,c,d}$, hence $\Lambda=\sum_ae_{aa}\otimes e_{aa}$ is the orthogonal projection on $span(e_a\otimes e_a)\subset\mathbb C^n\otimes\mathbb C^n$. Thus we have $\rho=\frac{n-1}{n}\delta_0+\frac{1}{n}\delta_1$, so $mn\rho=m(n-1)\delta_0+m\delta_1$, so $\pi_{mn\rho}=\pi_m$. Thus Theorem 3.2 says at $B=C=1$ that we have $m\tilde{W}\sim\pi_m$, as claimed.
\end{proof}

We have as well the following generalization, obtained by taking a tensor product:

\begin{proposition}
For a linear map of type $\varphi=id\otimes(tr(.)1\otimes id\otimes t\otimes (.)^\delta)$ we have $t\tilde{W}\sim law(\alpha-\beta)$, where $\alpha,\beta$ are free Poisson $(t(n_3\pm 1)/2)$, free, with $t=mn_1/n_2$.
\end{proposition}

\begin{proof}
We have $\Lambda=(n_2/n_1)(1\otimes P\otimes\Sigma\otimes Q)$, where $P,\Sigma,Q$ are the matrices appearing in the above proof, namely: $P$ is a rank 1 projection in $n_2^2$ dimensions, $\Sigma$ is the flip in $n_3^2$ dimensions, and $Q$ is a rank $n_4$ projection in $n_4^2$ dimensions. Since $P\otimes Q$ is a rank $n_4$ projection in $(n_2n_4)^2$ dimensions, the matrix $P\otimes\Sigma\otimes Q$ follows the following law:
$$\eta=\frac{n_2^2n_4-1}{n_2^2n_4}\,\delta_0+\frac{n_3-1}{2n_2^2n_3n_4}\,\delta_{-1}+\frac{n_3+1}{2n_2^2n_3n_4}\,\delta_1$$

Now since tensoring with $id$ doesn't change the law, it follows that $\rho$ is the measure obtained from $\eta$ by replacing the $\pm 1$ atoms by $\pm n_2/n_1$ atoms, so we have:
$$mn\rho=\frac{mn_1n_3(n_2^2n_4-1)}{n_2}\,\delta_0+\frac{mn_1(n_3-1)}{2n_2}\,\delta_{-n_2/n_1}+\frac{mn_1(n_3+1)}{2n_2}\,\delta_{n_2/n_1}$$

By using Theorem 3.1 and Theorem 3.2 it follows that $m\tilde{W}\sim law((n_2/n_1)(\alpha-\beta))$, where $\alpha,\beta$ are as in the statement, and by multiplying by $n_1/n_2$ we obtain the result.
\end{proof}

\section{The main result}

In this section we investigate the general, non-unital case. Let us first work out a generalization of Theorem 3.1, which will prove to be well-adapted to this case.

\begin{theorem}
Assume that $\varphi$ is self-adjoint. If for any $p\in\mathbb N$ and $\pi\in NC(p)$ we have
$$tr_{(1,1)}(\Lambda)tr_{(\pi,\gamma)}(\Lambda)
=tr_{(1,\pi\gamma^{-1})}(\Lambda)tr_{(\pi,\pi)}(\Lambda)$$
then $\delta m\tilde{W}\sim\pi_{mn\rho}\boxtimes\nu$, with $\rho=law(\Lambda)$, $\nu=law(D)$, $\delta=tr(D)$, where $D=\varphi(1)$.
\end{theorem}

\begin{proof}
The entries of $D=\varphi(1)$ are given by $D_{bd}=\sum_a\Lambda_{ab,ad}$, so we get:
$$tr_\sigma(D)=\sum_{e_r}\prod_sD_{e_re_{\sigma(r)}}=\sum_{ae_r}\Lambda_{ae_r,ae_{\sigma(r)}}=tr_{(1,\sigma)}(\Lambda)$$

In particular at $\sigma=1$ we obtain $\delta^p=tr_{(1,1)}(\Lambda)$, so Theorem 2.4 gives:
$$\lim_{d\to\infty}(\mathbb E\circ tr)((\delta m\tilde{W})^p)=\sum_{\pi\in NC(p)}(mn)^{|\pi|}tr_{(1,1)}(\Lambda)tr_{(\pi,\gamma)}(\Lambda)$$

We use the following formula from \cite{nsp} for of a multiplicative free convolution:
$$\mathbb E((AB)^p)=\sum_{\pi\in NC(p)}\prod_{b\in\pi}\mathbb E_{\pi\gamma^{-1}}(B)\kappa_\pi(A)$$

We know from Theorem 1.6 that the free cumulants of $\pi_{mn\rho}$ are the moments of $mn\rho$, given by $\kappa_p=mn\cdot tr(\Lambda^p)=mn\cdot tr_{(\gamma,\gamma)}(\Lambda)$. Thus the abstract free cumulants of $\pi_{mn\rho}$ are the numbers $\kappa_\pi=(mn)^{|\pi|}tr_{(\pi,\pi)}(\Lambda)$. Together with the formula $tr_\sigma(D)=tr_{(1,\sigma)}(\Lambda)$ above, the above general formula from \cite{nsp} shows that the moments of $\pi_{mn\rho}\boxtimes\nu$ are:
$$M_p=\sum_{\pi\in NC(p)}(mn)^{|\pi|}tr_{(1,\pi\gamma^{-1})}(\Lambda)tr_{(\pi,\pi)}(\Lambda)$$

We conclude that the formula $\delta m\tilde{W}\sim\pi_{mn\rho}\boxtimes\nu$ is equivalent to:
$$\sum_{\pi\in NC(p)}(mn)^{|\pi|}tr_{(1,1)}(\Lambda)tr_{(\pi,\gamma)}(\Lambda)
=\sum_{\pi\in NC(p)}(mn)^{|\pi|}tr_{(1,\pi\gamma^{-1})}(\Lambda)tr_{(\pi,\pi)}(\Lambda)$$

With this equivalence in hand, the assertion in the statement is now clear.
\end{proof}

The above result suggests the following technical definition:

\begin{definition}
A map $\varphi$ and the corresponding matrix $\Lambda$ are called ``multiplicative'' if
$$tr_{(1,1)}(\Lambda)tr_{(\pi,\gamma)}(\Lambda)
=tr_{(1,\pi\gamma^{-1})}(\Lambda)tr_{(\pi,\pi)}(\Lambda)$$
for any $p\in\mathbb N$ and $\pi\in NC(p)$.
\end{definition}

Observe that in this definition one can replace the normalized traces $tr$ by the unnormalized traces $Tr$: this follows indeed from $Tr_{(\pi,\sigma)}(\Lambda)=n^{|\pi|+|\sigma|}tr_{(\pi,\sigma)}(\Lambda)$.

As already mentioned, this definition is a technical one, coming straight from Theorem 4.1, which in turn comes from abstract algebraic manipulations. There is no simpler formulation of it, but in what follows we will try to have some understanding of it. 

Let us extend now Theorem 3.2 above. We have the following result here:

\begin{theorem}
The following types of maps and matrices are multiplicative:
\begin{enumerate}
\item $\varphi(A)=Tr(BA)C$, or $\Lambda=B\otimes C$, in the case $C=c1$.

\item $\varphi(A)=BAC$, or $\Lambda=|B\rangle\langle C|$, for any $B,C$.

\item $\varphi(A)=BA^tC$, or $\Lambda=\mathrm{SWAP}_{BC}$, in the case $BC=c1$.

\item $\varphi(A)=xA^\delta$, or $\Lambda=\mathrm{Center}_x$, in the case $x=c1$.
\end{enumerate}

\noindent In addition, the set of multiplicative maps and matrices is stable by tensor products.
\end{theorem}

\begin{proof}
(1,3,4) follow from Theorem 3.2, and the last assertion follows from the proof of Theorem 3.2. For proving (2) we use a formula found in the proof of Theorem 3.2:
$$Tr_{(\pi,\sigma)}(\Lambda)=Tr_{\pi\sigma^{-1}}(BC)$$

In terms of normalized traces, and with $D=BC$, we get:
$$tr_{(\pi,\sigma)}(\Lambda)=n^{|\pi\sigma^{-1}|-|\pi|-|\sigma|}tr_{\pi\sigma^{-1}}(D)$$

In particular we have the following formulae:
\begin{eqnarray*}
tr_{1,1}(\Lambda)&=&n^{p-p-p}tr_1(D)=n^{-p}tr_1(D)\\
tr_{\pi,\gamma}(\Lambda)&=&n^{(p+1-|\pi|)-|\pi|-1}tr_{\pi\gamma^{-1}}(D)=n^{p-2|\pi|}tr_{\pi\gamma^{-1}}(D)\\
tr_{1,\pi\gamma^{-1}}(\Lambda)&=&n^{(p+1-|\pi|)-p-(p+1-|\pi|)}tr_{\pi\gamma^{-1}}(D)=n^{-p}tr_{\pi\gamma^{-1}}(D)\\
tr_{\pi,\pi}(\Lambda)&=&n^{p-|\pi|-|\pi|}tr_1(D)=n^{p-2|\pi|}tr_1(D)
\end{eqnarray*}

By multiplying we obtain the formula in the statement.
\end{proof}

Using the graphical notation from \cite{cn1}, the diagrams for the maps of type (1-4) in Theorem 4.3 are some very simple ones, namely those in Figure \ref{fig:simple-Lambda}.

\begin{figure}
\centering
\subfigure[]{\includegraphics{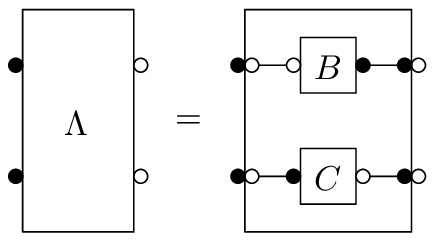}}\quad 
\subfigure[]{\includegraphics{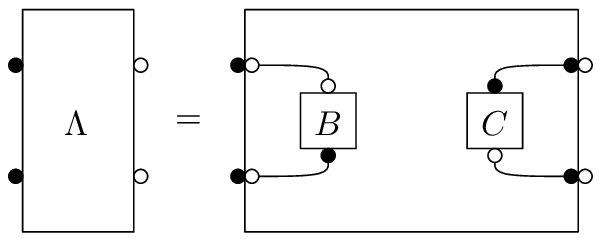}}\\
\subfigure[]{\includegraphics{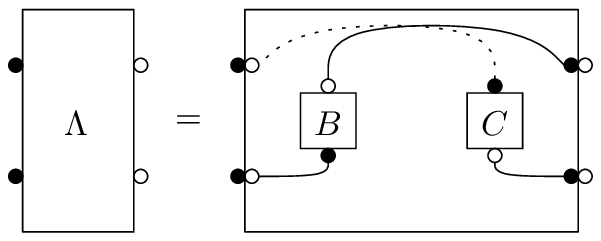}}\quad
\subfigure[]{\includegraphics{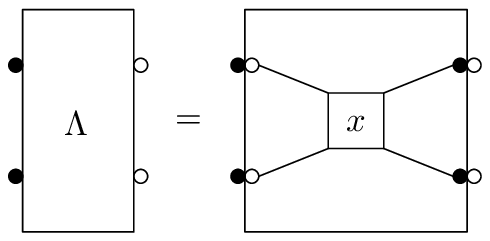}}
\caption{Special types of matrices $\Lambda$ whose mixed moments factorize properly.}
\label{fig:simple-Lambda}
\end{figure}

So, let us try to combine the diagrams in Figure 2 in a ``planar'' way. We call ``generalized ${\rm Center}_c$ diagram'' the diagram $\Lambda$ having $k$ left legs and $l$ right legs, given by:
$$\Lambda_{i_1\ldots i_k,j_1\ldots j_l}=c\cdot\delta_{i_1,\ldots,i_k,j_1,\ldots,j_l}$$

The following planar-categorical definition is inspired from \cite{jon}, \cite{cn1}:

\begin{definition}
The ``strings and beads'' operad $\mathcal P$ is defined as follows:
\begin{enumerate}
\item The elements are the matrices $\Lambda\in (M_{n_1}(\mathbb C)\otimes\ldots\otimes M_{n_k}(\mathbb C))^{\otimes 2}$, represented by vertical boxes with $2k$ left legs and $2k$ right legs.

\item The legs are colored by the values of the corresponding $n_i$ numbers, and the tensor product and composition operations have to match colors.

\item Inside the box we have ``strings'' joining the $2k+2k$ legs, i.e. we have a partition of $2k+2k$ elements into even blocks $\pi\in P_{even}(2k,2k)$, with colors matching.

\item Each string can be decorated with ``beads'', i.e. with usual matrices $A\in M_n(\mathbb C)$, where $n$ is the color of the string.

\item The ``multi-leg'' strings, representing blocks of $\pi\in P_{even}(2k,2k)$ having size $\geq 4$, have at the multi-crossing a ``small bead'', of generalized ${\rm Center}_c$ type. 
\end{enumerate}
\end{definition}

In other words, the matrices $\Lambda\in (M_{n_1}(\mathbb C)\otimes\ldots\otimes M_{n_k}(\mathbb C))^{\otimes 2}$ will be represented as diagrams having the left and right sequences of $2k$ points marked $n_1,\ldots,n_k,n_k,\ldots,n_1$, from top to bottom. In order to view such a matrix as $\Lambda\in M_{n_1\ldots n_k}(\mathbb C)\otimes M_{n_k\ldots n_1}(\mathbb C)$, we will simply ``compact the blocks'', i.e. we will view $\Lambda$ as a diagram between 2 left points and 2 right points, with both 2-series labeled $n_1\ldots n_k,n_k\ldots n_1$, from top to bottom.

We know from Theorem 4.3 that the 4 types of diagrams appearing there, as well as their tensor products, are multiplicative in the sense of Definition 4.2. The challenging question is to determine all the elements of $\mathcal P$ which are multiplicative.

Let us call ``through'' strings of $\Lambda\in\mathcal P$ the strings of the corresponding partition, joining left and right points. With this convention, we have the following result:

\begin{proposition}
Assume that $\Lambda\in\mathcal P$ is self-adjoint, and that its through strings are not decorated with beads. Then $\Lambda$ is multiplicative.
\end{proposition}

\begin{proof}
The element $\Lambda$ consists by definition of a partition $x\in P_{even}(2k,2k)$, decorated with beads. Since $\Lambda$ is self-adjoint, $x$ is ``symmetric'', in the sense that it is invariant under the reflection with respect to the vertical middle axis. We decompose $x=\sqcup_{i\in I}b_i$, where $b_i$ are blocks or pairs of blocks of $x$, chosen symmetric, and with $|I|$ maximal.

For any such ``symmetric block'' $b_i$ we denote by $\Lambda_i\in\mathcal P$ the element obtained from $\Lambda$ by keeping $b_i$ with its beads, and by completing with (undecorated) horizontal strings.

It follows from definitions that we have:
$$Tr_{(\pi,\sigma)}(\Lambda)=\prod_{i\in I}Tr_{(\pi,\sigma)}(\Lambda_i)$$

This formula shows that the multiplicativity condition splits over the symmetric blocks. Now for the symmetric blocks consisting of pairings, the result follows from Theorem 4.3 above. So, it remains to prove that the symmetrization $\tilde{\Lambda}$ of any generalized ${\rm Center}_c$ type partition $\Lambda$ is multiplicative. By adding an extra string if needed we can assume  that $\Lambda$ has an even number of legs at left and at right, and we have two cases here:

(1) $\tilde{\Lambda}=\Lambda$. Here the formula of $\Lambda$ is as follows, with $m\in\{2k-1,2k\}$:
$$\Lambda_{i_1\ldots i_{2k},j_1\ldots j_{2k}}=c\cdot\delta_{i_1,\ldots,i_m,j_1,\ldots,j_m}\delta_{i_{m+1}\ldots i_{2k},j_{m+1}\ldots j_{2k}}$$

With these notations, the mixed trace is given by:
\begin{eqnarray*}
Tr_{(\pi,\sigma)}(\Lambda)
&=&\sum_{i_1^r\ldots i_k^r}\prod_l\Lambda_{i_1^l\ldots i_{2k}^l,i_1^{\pi(l)}\ldots i_k^{\pi(l)}i_{k+1}^{\sigma(l)}\ldots i_{2k}^{\sigma(l)}}\\
&=&\sum_{i_1^r\ldots i_k^r}\prod_lc\cdot
\delta_{i_1^l,\ldots,i_m^l,i_1^{\pi(l)},\ldots,i_k^{\pi(l)},i_{k+1}^{\sigma(l)},\ldots i_m^{\sigma(l)}}
\delta_{i_{m+1}^l\ldots i_{2k}^l,i_{m+1}^{\sigma(l)}\ldots i_{2k}^{\sigma(l)}}\\
&=&c^p\cdot n^{|\pi\wedge\sigma|}\cdot n^{(2k-m)|\sigma|}
\end{eqnarray*}

Now since both $n^{|\pi\wedge\sigma|}$ and $n^{|\sigma|}$ are multiplicative, this gives the result.

(2) $\tilde{\Lambda}\neq\Lambda$. Here the formula of $\Lambda$ is as follows, with $m\in\{2k-1,2k\}$:
\begin{eqnarray*}
\tilde{\Lambda}_{i_1\ldots i_{2k},j_1\ldots j_{2k}}
&=&c^2\cdot\delta_{i_1,\ldots,i_s,i_{k+1},\ldots,i_{k+t},j_{s+1},\ldots,j_k,j_{k+t+1}\ldots,j_m}\\
&&\ \ \ \ \ \delta_{i_{s+1},\ldots,i_k,i_{k+t+1},\ldots,i_m,j_1,\ldots,j_s,j_{k+1}\ldots,j_{k+t}}\\
&&\ \ \ \ \ \delta_{i_{m+1}\ldots i_{2k},j_{m+1}\ldots j_{2k}}
\end{eqnarray*}

With these notations, the mixed trace is given by:
\begin{eqnarray*}
Tr_{(\pi,\sigma)}(\Lambda)
&=&\sum_{i_1^r\ldots i_k^r}\prod_lc\cdot
\delta_{i_1^l,\ldots,i_s^l,i_{k+1}^l,\ldots,i_{k+t}^l,i_{s+1}^{\pi(l)},\ldots,i_k^{\pi(l)},i_{k+t+1}^{\sigma(l)},\ldots i_m^{\sigma(l)}}\\
&&\ \ \ \ \ \ \ \ \ \ \ \ \ \ \delta_{i_{s+1}^l,\ldots,i_k^l,i_{k+t+1}^l,\ldots,i_m^l,i_1^{\pi(l)},\ldots,i_s^{\pi(l)},i_{k+1}^{\sigma(l)},\ldots i_{k+t}^{\sigma(l)}}\\
&&\ \ \ \ \ \ \ \ \ \ \ \ \ \ \delta_{i_{m+1}^l\ldots i_{2k}^l,i_{m+1}^{\sigma(l)}\ldots i_{2k}^{\sigma(l)}}
\end{eqnarray*}

By looking at the indices we obtain $Tr_{(\pi,\sigma)}(\Lambda)=c^2\cdot n^{2|\pi\wedge\sigma|}\cdot n^{(2k-m)|\sigma|}$, and since both the quantities $n^{|\pi\wedge\sigma|}$ and $n^{|\sigma|}$ are multiplicative, we are done.
\end{proof}

We can state and prove now our main result:

\begin{theorem}
Let $W$ be a complex Wishart matrix of parameters $(dn,dm)$, and let $\varphi:M_n(\mathbb C)\to M_n(\mathbb C)$ be a linear map coming from ``strings and beads'' diagram $\Lambda\in\mathcal P$, which is symmetric and has undecorated through strings. Then the $d\to\infty$ asymptotic eigenvalue distribution of the block-modified Wishart matrix $\tilde{W}=(id\otimes\varphi)W$ is given by $\delta m\tilde{W}\sim\pi_{mn\rho}\boxtimes\nu$, with $\rho=law(\Lambda)$, $\nu=law(D)$, $\delta=tr(D)$, where $D=\varphi(1)$.
\end{theorem}

\begin{proof}
We know from Proposition 4.5 that $\Lambda$ is multiplicative in the sense of Definition 4.2. Thus Theorem 4.1 applies, and gives the result.
\end{proof}

Observe that Theorem 4.6 doesn't fully cover Theorem 4.3. The problem is that, when trying to fully combine the diagrams in Theorem 4.3, our formula $\delta m\tilde{W}\sim\pi_{mn\rho}\boxtimes\nu$ seems to need a kind of substantial upgrade, with more quantities involved on the right. This kind of upgrade is also suggested by the various questions raised in the introduction, but finding it is of course a quite difficult problem, that we would like to raise here.

\end{document}